\documentclass{article}

\def\ParSkip{} 
\usepackage{amssymb,amsmath,amsthm,bbm}
\usepackage{verbatim,float,url,dsfont}
\usepackage{graphicx,subcaption,psfrag}
\usepackage{algorithm,algorithmic}
\usepackage{mathtools,enumitem}
\usepackage{multirow}
\usepackage{ragged2e}
\usepackage{xr-hyper}
\usepackage{array}

\usepackage[utf8]{inputenc} 
\usepackage[T1]{fontenc}    
\usepackage{booktabs}       
\usepackage{nicefrac}         
\usepackage{microtype}      

\usepackage[margin=1in]{geometry}
\usepackage[round]{natbib}
\usepackage[colorlinks=true,citecolor=blue,urlcolor=blue,linkcolor=blue]{hyperref}

\ifdefined\TimesFont 
\usepackage{times} 
\fi

\ifdefined\ParSkip 
\usepackage{parskip} 
\makeatletter
\def\thm@space@setup{%
  \thm@preskip=\parskip \thm@postskip=0pt
}
\makeatother
\fi

\newtheorem{theorem}{Theorem}

\newtheorem{corollary}{Corollary}
\newtheorem{proposition}{Proposition}
\theoremstyle{definition}
\newtheorem{remark}{Remark}

\newtheorem*{assumption*}{\assumptionnumber}
\providecommand{\assumptionnumber}{}


\makeatletter
\newcommand*\rel@kern[1]{\kern#1\dimexpr\macc@kerna}
\newcommand*\widebar[1]{%
  \begingroup
  \def\mathaccent##1##2{%
    \rel@kern{0.8}%
    \overline{\rel@kern{-0.8}\macc@nucleus\rel@kern{0.2}}%
    \rel@kern{-0.2}%
  }%
  \macc@depth\@ne
  \let\math@bgroup\@empty \let\math@egroup\macc@set@skewchar
  \mathsurround\z@ \frozen@everymath{\mathgroup\macc@group\relax}%
  \macc@set@skewchar\relax
  \let\mathaccentV\macc@nested@a
  \macc@nested@a\relax111{#1}%
  \endgroup
}
\makeatother

\DeclareMathOperator*{\minimize}{minimize}
\DeclareMathOperator*{\maximize}{maximize}

\DeclareMathOperator{\st}{subject\,\,to}

\DeclareMathOperator{\supp}{supp}

\DeclareMathOperator{\conv}{conv}

\def\E{\mathbb{E}}
\def\P{\mathbb{P}}
\def\R{\mathbb{R}}

\def\T{\mathsf{T}}

\def\hf{\hat{f}}

\def\hbeta{\hat{\beta}}

\def\th{^{\textnormal{th}}}

\def\cP{\mathcal{P}}
\def\cS{\mathcal{S}}

\def\cX{\mathcal{X}}
\def\cY{\mathcal{Y}}

\usepackage{stmaryrd} 
\def\zbag{\lbag z \rbag}
\def\Zbag{\lbag Z \rbag}

\def\d{\mathsf{d}}
\def\cQ{\mathcal{Q}}
\def\iid{\mathrm{iid}}
\def\X{\mathsf{X}}
\def\Y{\mathsf{Y}}
\def\cond{\mathrm{cond}}
\def\exch{\mathrm{exch}}
\DeclareMathOperator{\TV}{TV}

\RaggedRight
\title{Conformal Prediction Through the Lens of Hypothesis Testing:
  Universality, Impossibility, and Optimality}
\author{Ryan J.\ Tibshirani$^1$, Rina Foygel Barber$^2$, Aaditya Ramdas$^3$} 
\date{$^1$University of California, Berkeley, $^2$University of Chicago,
  $^3$Stanford University}

\begin{document}
\maketitle

\begin{abstract}
The connections between conformal prediction and permutation tests are already
widely-known in the literature. Some authors motivate conformal prediction by 
saying that it computes a permutation p-value for the hypothesis $H_0 : Y_{n+1}
= y$, and then inverts this to form a prediction set for $Y_{n+1}$ (i.e.,
accepts all values $y$ into the prediction set for which the p-value is
large). In this paper, we examine an alternative view, which is less well-known:
we again cast conformal prediction via the inversion of a permutation test, but
for the null of exchangeability of the joint distribution of the $n+1$
samples. This change in perspective, while simple, adheres more closely to
traditional formalization in hypothesis testing, which offers several
benefits. First, we use the duality between conformal sets and testing to show
that foundational universality and impossibility results in the conformal
prediction literature can be reproduced directly using classical hypothesis
testing theory (due to Neyman, Lehmann, Scheff{\'e}, Kraft, Le Cam, and
others). Furthermore, we show that an optimality result for conformal prediction
can be derived using standard Neyman-Pearson theory: for any joint distribution
of the covariates and response $X,Y$, and any sample size, the optimal method
for prediction sets---delivering the most efficient set among all methods
with valid coverage for exchangeable distributions---is a conformal predictor
whose score is the inverse conditional density of $Y|X$.      
\end{abstract}

\section{Introduction} 

Given independent and identically distributed (i.i.d.) samples
$Z_1,Z_2,\dots,Z_{n+1}$, where each $Z_i = (X_i,Y_i)$, many powerful tools have
been developed in statistics and machine learning over the years that enable us
to predict $Y_{n+1}$ based on $X_{n+1}$, and the first $n$ samples
$Z_1,\dots,Z_n$. Conformal prediction goes one step further: it is a framework
for producing a prediction set for $Y_{n+1}$ based on $X_{n+1}$ and
$Z_1,\dots,Z_n$. That is, conformal prediction is a general recipe for using
$X_{n+1}$ and $Z_1,\dots,Z_n$ to produce a set $C_n(X_{n+1})$ which satisfies, 
for a prespecified error level $\alpha \in [0,1]$,        
\begin{equation}
\label{eq:coverage}
\P\{Y_{n+1} \in C_n(X_{n+1})\} \geq 1-\alpha.
\end{equation}
Note that the probability statement in \eqref{eq:coverage} is over everything
that is random: $Z_1,\dots,Z_{n+1}$, and assumes nothing about the joint
distribution of each $(X_i,Y_i)$.       

Conformal prediction is typically at its best when the practitioner already has
a prediction model in mind that (they believe) works well for the problem at
hand, and a metric in mind which can be used to measure performance. In other
words, conformal prediction does not attempt to re-solve the prediction problem  
from scratch, but rather, it is a framework that is fundamentally about 
leveraging existing predictive capabilities, and carefully re-orienting them in
order to address the goal of uncertainty quantification.         

To be more precise, let $A$ be an algorithm which takes a sequence
$(z_1,\dots,z_k)$, of arbitrary length $k \geq 1$, and returns a predictor
\smash{$\hf = A(z_1,\dots,z_k)$} such that \smash{$\hf(x)$} is a prediction of
the value of $Y_i$ we expect to see when $X_i = x$. Also, let $\ell$ be a loss
function, where a lower value of \smash{$\ell(y, \hf(x))$} indicates that  
\smash{$\hf(x)$} is a more accurate prediction of $y$. With this notation in
place, conformal prediction works as follows: we first pick a \emph{trial value}
$y$ for $Y_{n+1}$, and we then train $A$ on the synthetically-augmented
sequence:   
\[
Z^y = (Z_1, \dots, Z_n, (X_{n+1}, y)),
\]
to produce a predictor $A(Z^y)$. The loss $\ell$ is used to measure the quality 
of these predictions, i.e., we compute what are known as conformal scores,  
\begin{equation}
\label{eq:cp_loss}
s^y_i = 
\begin{cases}
\ell(Y_i, A(Z^y)(X_i)) & \text{if $i \leq n$}, \\
\ell(y, A(Z^y)(X_{n+1})) & \text{if $i = n+1$}.
\end{cases}
\end{equation}
The conformal set $C_n(X_{n+1})$ is loosely defined as the set of all trial
values $y$ for which the last score \smash{$s^y_{n+1}$} is not ``too large''
relative to the other scores; when this happens the score \smash{$s^y_{n+1}$}
can be intuitively understood to ``conform'' to the behavior of the other
scores. Precisely,
\begin{equation}
\label{eq:cp_set}
C_n(X_{n+1}) = \Big\{ y : \text{$s^y_{n+1} \leq$ the $\lceil (n+1) (1-\alpha)
  \rceil\th$ smallest of $s^y_1, \dots, s^y_{n+1}$} \Big\}.   
\end{equation}
The standard guarantee backing conformal prediction is given next. It assumes an   
exchangeable sequence of data, which is less restrictive than assuming an
i.i.d.\ sequence, as we have done above.          

\begin{theorem}[Validity]
\label{thm:cp_validity}
Let $A$ be an algorithm which does not depend on the order of its input: namely,
for any $z_1,\dots,z_{n+1}$, it holds that
\begin{equation}
\label{eq:alg_symmetric}
A(z_1,\dots,z_{n+1}) = A(z_{\sigma(1)}, \dots, z_{\sigma(n+1)}), \quad \text{for
  any $\sigma \in \cS_{n+1}$},
\end{equation}
where $\cS_{n+1}$ is the set of permutations of $1,\dots,n+1$. Let $\ell$
be an arbitrary loss function. If $(Z_1,\dots,Z_{n+1})$ has an exchangeable
distribution, then the set defined in \eqref{eq:cp_loss}, \eqref{eq:cp_set}  
satisfies \eqref{eq:coverage}.      
\end{theorem}

\begin{remark}\it
\label{rem:cp_validity}
The scores do not need to be defined using an algorithm and a loss function, as
in \eqref{eq:cp_loss}. In fact, we can consider scores more generally defined by      
\begin{equation}
\label{eq:cp_score}
s^y_i = 
\begin{cases}
s((X_i,Y_i); Z^y) & \text{if $i \leq n$}, \\
s((X_{n+1}, y); Z^y) & \text{if $i = n+1$},
\end{cases}
\end{equation}
for a score function $s$ satisfying a generalization of
\eqref{eq:alg_symmetric}: for any $z = (z_1,\dots,z_{n+1})$ and any  $(x,y)$, 
\begin{equation}
\label{eq:score_symmetric} 
s((x,y); z) = s((x,y); z_\sigma), \quad \text{for any $\sigma \in \cS_{n+1}$},   
\end{equation}
where we denote \smash{$z_\sigma = (z_{\sigma(1)}, \dots, z_{\sigma(n+1)})$}. 
Then under the symmetry property \eqref{eq:score_symmetric}, the conformal set 
defined using \eqref{eq:cp_set}, \eqref{eq:cp_score} achieves coverage
\eqref{eq:coverage}. We will provide a proof of this in Section \ref{sec:cp_perm}. 
\end{remark}

Conformal prediction was pioneered by Vladimir Vovk and collaborators in the
1990s. In the last ten or so years, there has been an explosion of interest in
conformal prediction from the statistics and machine learning communities. The
literature on conformal and related topics is now vast, and we do not attempt to
provide an overview. Instead, we refer to the books by
\citet{vovk2022algorithmic, angelopoulos2025theoretical}, which provide an 
excellent treatment of the foundations as well as many of the newer
developments. In recent work \citep{barber2026unifying}, we showed that
conformal prediction and a number of recent extensions 
can be cast in a unified framework which uses the language of hypothesis
testing. This helps us understand how and why these methods work, and fluidly
leads to new combinations and generalizations.    

The current paper is not about generalizations of conformal prediction, but
about the basic method itself: we return to the connection between conformal
prediction and hypothesis testing from first principles, and examine what the
latter can teach us about the former. We find this to be fruitful in three
directions. We show that well-known results about \emph{universality} of
conformal prediction (Section \ref{sec:universality}), and \emph{impossibility}
of conditional coverage (Section \ref{sec:impossibility}), are consequences of
classical results in hypothesis testing (due to Neyman, Lehmann, Scheff{\'e},
Kraft, Le Cam, and others). Furthermore, the connections to hypothesis testing
enable us to draw conclusions about \emph{optimality} (Section
\ref{sec:optimality}): as an application of Neyman-Pearson theory, we show that
for any joint distribution of the covariates and response $X,Y$, and any sample
size $n$, the optimal method---delivering the most efficient prediction set in
an average sense, among all methods which have valid coverage over exchangeable
distributions---is a conformal predictor whose score function is the inverse
conditional density of $Y|X$.

Before covering these results, we begin by describing the general duality
between prediction sets and hypothesis tests, and explain how this connects
conformal prediction to permutation tests in particular.

\section{Prediction sets and hypothesis tests}

The duality between confidence sets and hypothesis tests, which dates back to 
\citet{neyman1937outline}, is a standard and widely-used result in classical
statistical inference. Less well-known perhaps is the parallel duality
between prediction sets and hypothesis tests. This dates back to at least
\citet{cox1975prediction}, but related ideas were around much earlier, such as 
tolerance regions \citep{wilks1941determination} and fiducial inference
\citep{fisher1935fiducial}. 

To fix notation, let $Z = (Z_1,\dots,Z_{n+1})$, and consider a null hypothesis
$H_0 : Z \sim P$. Let $\phi$ be a valid level $\alpha$ test for $H_0$, i.e., it
satisfies $\phi(z) \in \{0,1\}$ for each $z$, and
\[
\E_P[\phi(Z)] = \P_P\{\phi(Z) = 1\} \leq \alpha.
\]
We can invert this to form a prediction set for $Z_{n+1}$ based on
$Z_1,\dots,Z_n$, with coverage under $P$: letting 
\begin{equation}
\label{eq:gen_inversion}
C_n = \{ z : \phi(Z_1,\dots,Z_n,z) = 0 \}, 
\end{equation}
it follows that $\P_P\{Z_{n+1} \in C_n\} = \P_P\{\phi(Z) = 0\} \geq
1-\alpha$. The same construction extends to the setting in which each $Z_i =
(X_i,Y_i)$, and we seek a prediction set for $Y_{n+1}$ based on $X_{n+1}$ and 
$Z_1,\dots,Z_n$: letting   
\begin{equation}
\label{eq:gen_inversion_x}
C_n(X_{n+1}) = \{ y : \phi(Z_1,\dots,Z_n,(X_{n+1},y)) = 0 \},
\end{equation}
it again follows that $\P_P\{Y_{n+1} \in C_n(X_{n+1})\} = \P_P\{\phi(Z) = 0\}
\geq 1-\alpha$. We describe a few more extensions below and gather these into a
proposition for easy reference.   

\begin{proposition}
\label{prop:duality}
Let $\cP$ be a class of distributions, and consider the composite null
hypothesis $H_0: Z \sim P$ for $P \in \cP$. If $\phi$ is valid level $\alpha$
test for $H_0$,    
\[
\E_P[\phi(Z)] \leq \alpha, \quad \text{for all $P \in \cP$},
\]
then $C_n$ in \eqref{eq:gen_inversion} and $C_n(X_{n+1})$ in
\eqref{eq:gen_inversion_x} are both valid prediction sets with coverage
$1-\alpha$ over the class $\cP$,  
\begin{alignat*}{2}
\P_P\{Z_{n+1} \in C_n\} &\geq 1-\alpha, \quad && \text{for all $P \in \cP$}, \\
\P_P\{Y_{n+1} \in C_n(X_{n+1})\} &\geq 1-\alpha, \quad && \text{for all $P \in
  \cP$}. 
\end{alignat*}
Finally, if $\phi$ is an exact test, satisfying $\E_P[\phi(Z)] = \alpha$ for
each $P \in \cP$, then the inequalities above become equalities: these sets have
exact coverage $1-\alpha$ across $\cP$. 
\end{proposition}

Next, we work through some examples.

\subsection{Ordinary least squares regression and t-tests} 

To demonstrate the above duality, we turn to one of the most familiar settings
in statistics: ordinary least squares regression. Let $Z = (Z_1,\dots,Z_{n+1})
\sim P$ have i.i.d.\ entries, where each $Z_i = (X_i,Y_i)$ and 
\begin{equation}
\label{eq:linear_model} 
X_i \sim P_X, \quad Y_i|X_i \sim N(X_i^\T \beta, \sigma^2).
\end{equation}
Following the usual notation in regression, we denote by $\X \in \R^{n \times d}$ 
and $\Y \in \R^n$ the covariate matrix and response vector, respectively, 
collected over the first $n$ samples. Assuming that $n>d$ and $\X^\T \X$ is
almost surely invertible, we denote the ordinary least squares coefficients by      
\[
\hbeta = (\X^\T \X)^{-1} \X^\T \Y.
\]
Furthermore, denote $\hat\sigma^2 = \|\Y - \X\hbeta\|_2^2 / (n-d)$, and define 
\begin{equation}
\label{eq:t_stat}
T(Z) = \frac{Y_{n+1} - X_{n+1}^\T \hbeta}{\hat\sigma \sqrt{1 + X_{n+1}^\T 
(\X^\T \X)^{-1} X_{n+1}}}.
\end{equation}
By a standard calculation, this is a pivotal statistic, and has distribution
$T(Z) \sim t_{n-d}$, where $t_{n-d}$ denotes the t-distribution with $n-d$
degrees of freedom. Define the test 
\[
\phi(Z) = 1\{ |T(Z)| > t_{n-d,1-\alpha/2}\},
\] 
where $t_{n-d,1-\alpha/2}$ denotes the level $1-\alpha/2$ quantile of $t_{n-d}$,
and let $\cP = \{ P_{X,Y}^{n+1} : \text{$P_{X,Y}$ is of the form
  \eqref{eq:linear_model}} \}$. Then by construction $\phi$ is an exact level  
$\alpha$ test over $\cP$, i.e., $\E_P[\phi(Z)] = \alpha$, for all $P \in \cP$.   

Pursuing the inversion scheme \eqref{eq:gen_inversion_x} leads to      
\begin{align}
\nonumber
C_n(X_{n+1}) &= \Big\{ y : \big| T(Z_1,\dots,Z_n,(X_{n+1},y)) \big| >
t_{n-d,1-\alpha/2} \Big\} \\
\label{eq:t_set}
&= \big[ X_{n+1}^\T \hbeta - \hat{s}_{X_{n+1}} \hspace{-2pt} \cdot
t_{n-d,1-\alpha/2}, \, X_{n+1}^\T \hbeta + \hat{s}_{X_{n+1}} \hspace{-2pt} \cdot
t_{n-d,1-\alpha/2} \big],  
\end{align}
where we abbreviate \smash{$\hat{s}_{X_{n+1}}^2 = \hat\sigma^2 (1 + X_{n+1}^\T
  (\X^\T \X)^{-1} X_{n+1})$}. The set in \eqref{eq:t_set} is the textbook
ordinary least squares prediction interval for $Y_{n+1}$ based on
$X_{n+1},Z_1,\dots,Z_n$. Its validity comes directly from the duality between
prediction sets and testing in Proposition \ref{prop:duality}: since $Y_{n+1}
\in C_n(X_{n+1}) \iff \phi(Z) = 0$, it holds that $\P_P\{Y_{n+1} \in
C_n(X_{n+1})\} = 1-\alpha$, for all $P \in \cP$.

More can be said. The statistic $T(Z)$ in \eqref{eq:t_stat} is in fact
\emph{conditionally} distributed as $t_{n-d}$ given $X_1,\dots,X_{n+1}$, which
means that $\phi$ is conditionally valid: $\E_P[\phi(Z) \,|\, X_1,\dots,X_{n+1}]
= \alpha$, for all $P \in \cP$. Thus, by the same logic, the prediction interval
in \eqref{eq:t_set} is also conditionally valid: 
\begin{equation}
\label{eq:t_set_conditional}
\P_P\Big\{ Y_{n+1} \in \big[ x^\T \hbeta - \hat{s}_x \hspace{-1pt} \cdot 
t_{n-d,1-\alpha/2}, \, x^\T \hbeta + \hat{s}_x \hspace{-1pt} \cdot 
t_{n-d,1-\alpha/2} \big] \,\big|\, \X, X_{n+1}= x \Big\} = 1-\alpha, 
\end{equation}
for all $x$, and $P \in \cP$. This turns out to be a special feature of the
parametric model in \eqref{eq:linear_model}, which admits a pivotal statistic
for fixed design (i.e., fixed $\X$ and $X_{n+1}$). In contrast, as we will see later
on (Section \ref{sec:impossibility}), analogous conditional coverage statements
cannot hold for more general methods which are valid over all i.i.d.\ sequences,
except for those producing trivially inefficient prediction sets. 

\subsection{Conformal prediction and permutation tests}    
\label{sec:cp_perm}

Now we turn to a nonparametric setting, which will remain our focus for the 
rest of the paper. Let $Z = (Z_1,\dots,Z_{n+1}) \sim P$, where $P$ is
exchangeable. Given a statistic $T$, define the permutation p-value  
\begin{equation}
\label{eq:perm_pval}
p = \frac{1}{(n+1)!} \sum_{\sigma \in \cS_{n+1}} 1\{ T(Z_\sigma) \geq T(Z) \}, 
\end{equation}
where recall $\cS_{n+1}$ is the set of all permutations of $1,\dots,n+1$, and we
write \smash{$Z_\sigma = (Z_{\sigma(1)}, \dots, Z_{\sigma(n+1)})$}. The level
$\alpha$ permutation test is then defined by  
\[
\phi(Z) = 1\{ p \leq \alpha \}.
\]
Let $\cP$ denote the class of all exchangeable distributions. Then by a standard 
result in permutation testing (e.g., Theorem 17.2.1 in
\citealt{lehmann2022testing}), this is a valid level $\alpha$ test over $\cP$,  
i.e., $\E_P[\phi(Z)] \leq \alpha$, for all $P \in \cP$. (This test can be made
exact by auxiliary randomization.)    

Let us inspect the inversion scheme \eqref{eq:gen_inversion_x}. Writing as
before $Z^y = (Z_1, \dots, Z_n, (X_{n+1}, y))$, this yields
\begin{equation}
\label{eq:perm_inversion}
C_n(X_{n+1}) = \bigg\{ y : \frac{1}{(n+1)!} \sum_{\sigma \in \cS_{n+1}} 1\{
T(Z^y_\sigma) \geq T(Z^y) \} > \alpha \bigg\}.
\end{equation}
Suppose further that we set $T(z) = s(z_{n+1}; z)$ for a score function $s$
satisfying the symmetry condition \eqref{eq:score_symmetric}. Then $T(Z_\sigma)
= s((X_i,Y_i); Z)$ for any permutation $\sigma$ such that $\sigma(n+1) = i$.  
Recalling the notation in \eqref{eq:cp_score} for conformal scores, the set in 
\eqref{eq:perm_inversion} therefore reduces to
\begin{equation}
\label{eq:perm_inversion_symmetric}
C_n(X_{n+1}) = \bigg\{ y : \frac{1}{n+1} \sum_{i=1}^{n+1} 1\{ s^y_i \geq
s^y_{n+1} \} > \alpha \bigg\}.
\end{equation}
This is a well-known equivalent form for the conformal prediction set in 
\eqref{eq:cp_set} (e.g., Chapter 2.2.3 of \citealt{vovk2022algorithmic}, or
Chapter 3.5.1 of \citealt{angelopoulos2025theoretical}). Consequently, the
validity of conformal prediction in Theorem \ref{thm:cp_validity} (and the
generalization in Remark \ref{rem:cp_validity} to arbitrary symmetric scores)
can again be derived from the duality in Proposition \ref{prop:duality}: as
$Y_{n+1} \in C_n(X_{n+1}) \iff \phi(Z) = 0$, it holds that $\P_P\{Y_{n+1} \in
C_n(X_{n+1})\}  \geq 1-\alpha$, for all exchangeable distributions $P$.

Of course, this duality carries over even in the general case in which $T$ does
not exhibit any special symmetry properties. (After all, validity of the
permutation test does not require any assumptions about the statistic $T$.)
Setting $T(z) = s(z_{n+1}; z)$, without any symmetry condition on $s$, the set
in \eqref{eq:perm_inversion} becomes 
\begin{equation}
\label{eq:perm_inversion_general}
C_n(X_{n+1}) = \bigg\{ y : \frac{1}{(n+1)!} \sum_{\sigma \in \cS_{n+1}} 1\big\{
s(Z^y_{\sigma(n+1)}; Z^y_\sigma) \geq s((X_{n+1},y); Z) \big\} > \alpha \bigg\}. 
\end{equation}
Proposition \ref{prop:duality} again implies this set has coverage:
$\P_P\{Y_{n+1} \in C_n(X_{n+1})\} \geq 1-\alpha$, for all exchangeable
$P$. Already, this says something interesting about conformal prediction:
symmetry of the score function is \emph{not} necessary for validity; it is
instead a computational device used to simplify the combinatorial nature of the
permutation p-value. In other words, exchangeability of the data (not
exchangeability of the scores) is all that is needed.

\subsection{What is new (and what is not)?}

The duality between hypothesis tests and prediction sets described at the start
of this section is not new. Although it is perhaps underappreciated, this idea
dates back at least 50 years to \citet{cox1975prediction}, or 85 years to
\citet{wilks1941determination}, depending on the interpretation. The connections
between conformal prediction and hypothesis testing are also not new. Vovk and
coauthors have traditionally defined conformal prediction sets by inverting
p-values; see, e.g., Chapter 2.2.3 of \citet{vovk2022algorithmic}, and also
\citet{vovk1999machine}, which marks the inception of the
literature. Furthermore, these authors frequently motivate conformal prediction
by drawing connections to the work of Gosset, Fisher, Neyman, and others in the 
early to mid 1900s; see, e.g., Chapter 13.3 of \citet{vovk2022algorithmic}, and
also \citet{vovk2008tutorial} for related discussion.

The statistics literature on conformal prediction has also made clear reference
to hypothesis testing, with \citet{lei2013distribution, lei2014distribution,
  lei2018distribution} being early instances of this. Indeed, in both
\citet{lei2014distribution, lei2018distribution}, it is explicitly
stated that the conformal prediction set in \eqref{eq:cp_set} can be viewed as 
inverting a test for the null hypothesis $H_0 : Y_{n+1} = y$. Connections
between conformal prediction and permutation/randomization testing are also  
drawn in \citet{dobriban2023joint, hoff2023bayes, zhang2023what,
  dobriban2025symmpi, koning2025fuzzy, barber2026unifying,
  nair2026randomization}, and Chapter 3.5.1 of 
\citet{angelopoulos2025theoretical}, among others.      

What is new in the current paper?  To build toward our results in the coming
sections, we start by clarifying the formal connection between conformal
prediction and permutation testing (Section \ref{sec:cp_perm}), emphasizing that
conformal prediction inverts a permutation test for the null hypothesis of
\emph{exchangeability of the full data $Z = (Z_1,\dots,Z_{n+1})$}. To state the
connection once again: the conformal set \eqref{eq:cp_set} can be equivalently
written as 
\begin{multline*}
C_n(X_{n+1}) = \Big\{ y : \text{the level $\alpha$ permutation test with
  statistic $T(z) = s(z_{n+1}; z)$ applied to} \\ 
\text{data $(Z_1,\dots,Z_n,(X_{n+1},y))$ does not reject the null of
  exchangeability} \Big\}.   
\end{multline*}
With $P$ denoting the distribution of $Z$, and $\cP$ the class of exchangeable
distributions, the underlying null hypothesis here is $H_0: P \sim \cP$, which
maps cleanly onto traditional formalization in statistical inference.  

Our main contribution is to follow this connection to its logical conclusion,
and ask what hypothesis testing theory can teach us about conformal
prediction. This yields simple, streamlined proofs of results on universality
and impossibility (Sections \ref{sec:universality} and
\ref{sec:impossibility}), which are well-known in the conformal literature, as 
well as on optimality (Section \ref{sec:optimality}), which is comparatively new
and still emerging in the literature. Moreover, beyond these particular results,
our paper suggests a shift in perspective. Most of the conformal literature
referenced above treats facts about hypothesis testing as motivating principles,
and treats conformal theory as separate, but analogous. We present a different
view: these are not just analogies, and core facts about hypothesis testing
\emph{directly} translate to conformal prediction. Hence, the former provides a
language which \emph{precisely} helps us understand the latter.

\section{Universality}
\label{sec:universality}

Let $\phi$ be a valid level $\alpha$ test for a composite null hypothesis $H_0 :
Z \sim P$, $P \in \cP$, i.e., let $\phi$ satisfy
\begin{equation}
\label{eq:validity}
\E_P[\phi(Z)] \leq \alpha, \quad \text{for all $P \in \cP$}.
\end{equation}
Let $U = U(Z)$ be a sufficient statistic for $\cP$ (which means the distribution
of $Z|U$ does not depend on $P$). For $P \in \cP$, let $P_U$ denote the induced 
distribution on $U$, and similarly, let $\cP_U = \{ P_U : P \in \cP \}$ denote
the induced class. Then any test which satisfies\footnote{For a class $\cP$, a
  statement is said to hold $\cP$-almost everywhere if it holds except on a set
  $S$ with \mbox{$P(S) = 0$} for all \mbox{$P \in \cP$}.}
\begin{equation}
\label{eq:neyman_structure}
\E[\phi(Z) | U=u] \leq \alpha, \quad \text{for $\cP_U$-almost every $u$},    
\end{equation}
clearly satisfies \eqref{eq:validity}. Tests satisfying
\eqref{eq:neyman_structure} are said to have \emph{Neyman structure} with
respect to $U$ (see, e.g., Chapter 4.3 of \citealt{lehmann2022testing}). A
seminal result, which can be traced back to the work of
\citet{neyman1937outline, ghosh1948extension, hoel1948neymans,
  lehmann1950completeness, lehmann1955completeness, basu1955statistics}: 
if the class $\cP$ is rich enough, then Neyman structure
\eqref{eq:neyman_structure} is not only sufficient for validity
\eqref{eq:validity}, but also necessary.     

\begin{theorem}[Adaptation of Theorem 4.3.2 in \citealp{lehmann2022testing}] 
\label{thm:neyman_structure}
Let $U$ be a sufficient statistic for $\cP$. Assume that $U$ is \emph{boundedly
  complete} for $\cP$, which means that for any bounded function $f$,   
\begin{equation}
\label{eq:boundedly_complete_1s}
\E_{P_U}[f(U)] \leq 0 \;\, \text{for all $P_U \in \cP_U$} \quad \implies \quad  
f(u) \leq 0 \;\, \text{for $\cP_U$-almost every $u$}.  
\end{equation}
Then a necessary and sufficient condition for a test to be level $\alpha$ over
$\cP$ \eqref{eq:validity} is that it is of Neyman structure 
\eqref{eq:neyman_structure}. That is, marginal validity \eqref{eq:validity} and 
$U$-conditional validity \eqref{eq:neyman_structure} over $\cP$ are equivalent.    
\end{theorem}

To be clear, our treatment is a seemingly minor yet nontrivial adaptation of the
results in Chapter 4.3 of \citet{lehmann2022testing}. They focus on exact level
$\alpha$ tests, whereas we focus on valid (but possibly conservative) level
$\alpha$ tests. This requires modifying the definitions of Neyman structure and
bounded completeness. Notably, our definition \eqref{eq:boundedly_complete_1s}
is stronger than the traditional definition of bounded completeness, but is
suitable for our purposes in what follows. A simple proof of Theorem
\ref{thm:neyman_structure} is given in Appendix \ref{app:neyman_structure}. A
discussion of the differences in notions of completeness is given in Appendix
\ref{app:boundedly_complete}.  

To study the application of this theory to conformal prediction, we inspect the
case in which $\cP$ is the class of exchangeable distributions. First, note that 
a sufficient statistic for this class is the multiset or ``bag'' of elements of 
$Z = (Z_1,\dots,Z_{n+1})$, denoted \smash{$U = \Zbag$}.\footnote{When
  \smash{$Z \in \R^{n+1}$}, an equivalent formulation for the sufficient
  statistic here is to define \smash{$U = (Z_{(1)}, \dots, Z_{(n+1)})$},
  the vector of order statistics. This is more common in the classical
  literature on statistical inference. We use the multiset \smash{$U = \Zbag$},
  as in the conformal literature, because this extends seamlessly to arbitrary 
  unordered spaces.}  
For $Z \sim P$, with $P$ exchangeable, and \smash{$u = \zbag$},  
\begin{equation}
\label{eq:perm_dist}
Z | U=u \sim \frac{1}{(n+1)!} \sum_{\sigma \in \cS_{n+1}} \delta_{z_\sigma}.  
\end{equation}
The right-hand side is the uniform distribution over all permutations $z_\sigma$
of the given $n+1$ elements, and does not depend on $P$, which verifies
sufficiency. Next, we claim \smash{$U = \Zbag$} is boundedly complete  
\eqref{eq:boundedly_complete_1s} for the class $\cP$ of exchangeable
distributions. This is straightforward to check: let $f$ be such that
\smash{$\E_{P_U}[f(U)] \leq 0$} for all exchangeable $P$, fix \smash{$u =
  \zbag$}, and inspect \smash{$P = \frac{1}{(n+1)!} \sum_{\sigma \in \cS_{n+1}} 
  \delta_{z_\sigma}$}. This choice is exchangeable, and the induced
distribution on $U$ is $P_U = \delta_u$. By construction \smash{$\E_{P_U}[f(U)]
  = f(u)$}, and so $f(u) \leq 0$.    

Therefore, Theorem \ref{thm:neyman_structure} implies that any test $\phi$ which
is level $\alpha$ over the class of exchangeable distributions must have Neyman 
structure with respect to \smash{$U = \Zbag$}. Recalling \eqref{eq:perm_dist}, 
we have for \smash{$u = \zbag$},  
\[
\E[\phi(Z) | U=u] = \frac{1}{(n+1)!} \sum_{\sigma \in \cS_{n+1}} \phi(z_\sigma),   
\]
and so Neyman structure \eqref{eq:neyman_structure} in the current setting 
translates to  
\begin{equation}
\label{eq:neyman_structure_bag}
\frac{1}{(n+1)!} \sum_{\sigma \in \cS_{n+1}} \phi(z_\sigma) \leq \alpha, \quad
\text{for all $z$}.
\end{equation}
We can verify that the above condition is equivalent to $\phi$ having
permutation structure, i.e.,    
\begin{equation}
\label{eq:neyman_structure_perm}
\text{$\phi$ satisfies \eqref{eq:neyman_structure_bag}} \quad \iff \quad
\text{$\phi(Z) = 1\{p \leq \alpha\}$ for a p-value $p$ of the form 
  \eqref{eq:perm_pval}}.  
\end{equation}
This follows from a standard argument in the literature on permutation testing, 
given in Appendix \ref{app:neyman_structure_perm} for completeness. As valid
prediction sets are the result of inverting valid tests, and conformal
prediction sets are the result of inverting permutation tests, we have thus 
reproduced a well-known conformal universality result (cf.\ Proposition 2.9 of
\citealt{vovk2022algorithmic}, or Theorem 9.7 of 
\citealt{angelopoulos2025theoretical}).   

\begin{theorem}[Universality]
\label{thm:universality}
Let $Z = (Z_1,\dots,Z_{n+1}) \sim P$, and consider a mapping which takes 
$X_{n+1}$ and $Z_1,\dots,Z_n$ to a prediction set $C_n(X_{n+1})$. If
$\P_P\{Y_{n+1} \in C_n(X_{n+1})\} \geq 1-\alpha$ for all exchangeable $P$, then 
this prediction set must be of conformal type, i.e., of the form
\eqref{eq:perm_inversion_general} for a score function $s$.     

Furthermore, if the set $C_n(X_{n+1})$ is invariant to the ordering of
$Z_1,\dots,Z_n$, then it is of the form \eqref{eq:perm_inversion_symmetric} for 
scores \eqref{eq:cp_score} and a symmetric score function $s$.  
\end{theorem}

\begin{proof}
Write $C_n(X_{n+1}) = \{ y : \phi(Z_1,\dots,Z_n,(X_{n+1},y)) = 0 \}$ for a test
$\phi$, which note we can do without a loss of generality since we may always
take $\phi(Z) = 1\{Y_{n+1} \notin C_n(X_{n+1})\}$. The assumed 
coverage property now implies $\E_P[\phi(Z)] \leq \alpha$ for all exchangeable
$P$. By the arguments given above, this implies $\phi$ must have Neyman
structure \eqref{eq:neyman_structure_bag} with respect to \smash{$U = \Zbag$},
and hence $\phi(Z) = 1\{p \leq \alpha\}$ for a permutation p-value $p$ as in
\eqref{eq:perm_pval} and statistic $T$. Defining the score function by
$s(z_{n+1}; z) = T(z)$, we have shown $C_n(X_{n+1})$ is of the form
\eqref{eq:perm_inversion_general}, which proves the first claim in the theorem.       

For the second claim, if $C_n(X_{n+1})$ is invariant to the ordering of
$Z_1,\dots,Z_n$, then the same must be true of $T$ and $s$. Thus $s$ meets the
symmetry condition \eqref{eq:score_symmetric},\footnote{Note that the following 
  statements are all equivalent: a score $s$ satisfies
  \eqref{eq:score_symmetric}; 
  it may be written as \mbox{$s(z_{n+1}; \{z_1,\dots,z_{n+1}\})$}; 
  it may be written as \mbox{$s(z_{n+1}; \{z_1,\dots,z_n\})$}; 
  and hence \mbox{$s(z_{n+1}; z)$} is invariant to the order of
  $z_1,\dots,z_n$.} 
and by the arguments given in Section \ref{sec:cp_perm}, the set
\eqref{eq:perm_inversion_general} reduces to
\eqref{eq:perm_inversion_symmetric}. This completes the proof.     
\end{proof}

\section{Efficiency}
\label{sec:efficiency}

In what follows, we assume that $Z = (Z_1,\dots,Z_{n+1}) \sim P$ has i.i.d.\
entries, i.e., we assume $P$ has the product form \smash{$P = P_{X,Y}^{n+1}$}
for a joint distribution on the covariates and response
\smash{$P_{X,Y}$}. Letting \smash{$P_{X,Y} = P_X \times P_{Y|X}$}, note that we
can equivalently express \smash{$Z_i = (X_i,Y_i) \sim P_{X,Y}$} as \smash{$X_i
  \sim P_X$} and \smash{$Y_i|X_i \sim P_{Y|X=X_i}$}.        

Let $\cY$ denote the response space, and let $\mu$ be a measure on $\cY$. For
example, for a continuous response we could take $\cY = \R$ and take $\mu$ to be 
Lebesgue measure; for a categorical response $\cY$ would be discrete and we
could take $\mu$ to be counting measure. Given a prediction set mapping, which 
takes $X_{n+1},Z_1,\dots,Z_n$ to a set $C_n(X_{n+1})$, we define its
\emph{efficiency} by     
\[
\E_P[\mu(C_n(X_{n+1}))],
\]
the average measure of $C_n(X_{n+1})$, where the average is over all sources of
randomness: $X_{n+1},Z_1,\dots,Z_n$. Using the representation
\eqref{eq:gen_inversion_x} for $C_n(X_{n+1})$ as the inversion of a test $\phi$, 
observe that we can write
\begin{align}
\nonumber
\E_P[\mu(C_n(X_{n+1}))] 
&= \int \E_{P_{X,Y}^n}\big[ 1\{y \in C_n(x)\} \big] \, \d{P}_X(x) \, \d\mu(y) \\  
\label{eq:mu_integral}
&= \int \E_{P_{X,Y}^n}\big[ 1 - \phi(Z_1,\dots,Z_n,(x,y)) \big] \, \d{P}_X(x) \, 
  \d\mu(y).
\end{align}
Assuming $\mu(\cY) < \infty$, we can normalize $\mu$ to form a probability
measure, and define    
\begin{equation}
\label{eq:alternative_q}
Q = P_{X,Y}^n \times \Big( P_X \times \frac{\mu}{\mu(\cY)} \Big). 
\end{equation}
This allows us to express efficiency \eqref{eq:mu_integral} succinctly as 
\begin{equation}
\label{eq:mu_type_II}
\E_P[\mu(C_n(X_{n+1}))] = \mu(\cY) \cdot \E_Q[1-\phi(Z)].
\end{equation}
In other words, while we have seen that coverage of prediction sets is a
statement about type I error under the null hypothesis $P$ (Proposition 
\ref{prop:duality}), the formulation in \eqref{eq:mu_type_II} reveals that
efficiency is a statement about \emph{type II} error: the right-hand side is
precisely a constant (the total mass $\mu(\cY)$) times the type II error of
$\phi$ under the alternative hypothesis $Q$ in \eqref{eq:alternative_q}.    

In the next two sections, we will leverage the connection in
\eqref{eq:mu_type_II} to investigate different questions concerning
efficiency. This is a powerful reframing, as it allows us to translate such 
questions to ones about type I versus type II errors in hypothesis testing.    

\section{Impossibility}
\label{sec:impossibility}

In this section, we use the representation for efficiency developed in the
previous section to derive results on the impossibility of conditional coverage,
establishing that conformal prediction cannot offer nontrivial guarantees which 
are conditional---rather than marginal---on the covariate $X_{n+1}$. In
particular, we define the class \smash{$\cP_\iid = \{ P : P = P_{X,Y}^{n+1} \;\,
  \text{for some $P_{X,Y} = P_X \times P_{Y|X}$} \}$}, and we now consider
methods which satisfy the conditional property 
\begin{equation}
\label{eq:cond_validity1}
\P_P\{Y_{n+1} \in C_n(x) \,|\, X_{n+1}=x \} \geq 1-\alpha, \quad \text{for all
  $x \in \supp(P_X)$, and all $P \in \cP_\iid$},
\end{equation}
where $\supp(P_X)$ denotes the support of $P_X$.\footnote{For a distribution
  $P$, its support $\supp(P)$ is the smallest closed set $S$ such that 
  \mbox{$P(S) = 1$}. We work with the conditional coverage property in
  \eqref{eq:cond_validity1}, which holds at each \mbox{$x \in \supp(P_X)$},
  because it cleanly transforms to \eqref{eq:cond_validity2}. We could
  alternatively choose to define \eqref{eq:cond_validity1} as holding for
  $P_X$-almost every $x$, which would be more in line with the typical phrasing 
  of conditional coverage properties in the conformal literature; however, this
  requires slightly more sophisticated arguments which we prefer to avoid for
  simplicity.}   
We can recast this equivalently as follows. Let $\cX$ denote the covariate
space, and $\delta_x$ denote the point mass at $x \in \cX$ (it is the
distribution which places all it mass at $x$). Define the class  
\[
\cP_\cond = \Big\{ P : P = P_{X,Y}^n \times (\delta_x \times P_{Y|X=x}) \;\,
\text{for some $P_{X,Y} = P_X \times P_{Y|X}$ and $x \in \supp(P_X)$} \Big\}, 
\]
Then the conditional coverage property in \eqref{eq:cond_validity1} is
equivalent to   
\begin{equation}
\label{eq:cond_validity2}
\P_P\{Y_{n+1} \in C_n(X_{n+1})\} \geq 1-\alpha, \quad \text{for all $P \in
  \cP_\cond$},  
\end{equation}
which is marginal coverage over \smash{$\cP_\cond$}. 

For a given distribution \smash{$P = P_{X,Y}^{n+1}$}, and a given non-atom point
$x_0 \in \supp(P_X)$, we will study the limits of pointwise efficiency
\smash{$\E_P[\mu(C_n(x_0))]$} among all methods with conditional coverage over
\smash{$\cP_\iid$}, as in \eqref{eq:cond_validity1}, or equivalently, marginal
coverage over \smash{$\cP_\cond$}, as in \eqref{eq:cond_validity2}. As
demonstrated in Section \ref{sec:efficiency}, this can be posed as a question
about power. If $\mu(\cY) < \infty$, then analogous to \eqref{eq:alternative_q},
we can define  
\begin{equation}
\label{eq:alternative_q0}
Q_0 = P_{X,Y}^n \times \Big( \delta_{x_0} \times \frac{\mu}{\mu(\cY)} \Big),   
\end{equation}
and by arguments analogous to those leading to \eqref{eq:mu_type_II}, it holds
that 
\begin{equation}
\label{eq:mu_type_II_x0}
\E_P[\mu(C_n(x_0))] = \mu(\cY) \cdot \E_{Q_0}[1-\phi(Z)].
\end{equation}
Note that the right-hand side is a constant times the type II error of $\phi$
against $Q_0$, or one minus its power. Thus, our question becomes: what is
largest possible power among all tests valid over \smash{$\cP_\cond$}?    

We break our analysis below into two cases: the finite measure case $\mu(\cY) < 
\infty$ (which is needed in order for $Q_0$ to be well-defined), and the
infinite measure case (which requires a truncation argument). In either case, we
will rely on the following core result. This can be found in Theorem 5 of
\citet{kraft1955conditions}, however, Kraft attributes the result to Le
Cam. Theorem 5 in \citet{kraft1955conditions} actually establishes a stronger
statement than that transcribed below, under conditions on the distributions
$P,Q$. The weaker version we state in Theorem \ref{thm:kraft_le_cam} follows
directly from their proof without any such conditions. 

\begin{theorem}[Weak version of Theorem 5 in \citealp{kraft1955conditions}] 
\label{thm:kraft_le_cam}
Let $\cP$ and $\cQ$ be classes of distributions, and let $\phi$ be a valid level
$\alpha$ test over $\cP$, i.e., \smash{$\sup_{P \in \cP} \E_P[\phi(Z)] \leq
  \alpha$}. Then the power of $\phi$ over $\cQ$ is upper bounded by  
\[
\inf_{Q \in \cQ} \, \E_Q[\phi(Z)] \leq \alpha + \TV(\conv(\cP), \conv(\cQ)),  
\]
where $\conv(\cP), \conv(\cQ)$ denote the convex hulls of $\cP$, $\cQ$,
respectively, and $\TV(\conv(\cP), \conv(\cQ))$ denotes the total variation
distance between the hulls, i.e., the infimum of $\TV(P,Q) = \sup_S |P(S) -
Q(S)|$ over all $P \in \conv(\cP)$ and $Q \in \conv(\cQ)$.   
\end{theorem}

\subsection{Finite measure case}

In the finite measure case $\mu(\cY) < \infty$, the distribution $Q_0$ in
\eqref{eq:alternative_q0} is well-defined, and the arguments above already
bring us close to a result about the limits of efficiency of prediction set
methods which are conditionally valid, as in \eqref{eq:cond_validity1}. The key
realization is that    
\begin{equation}
\label{eq:tv_zero}
\TV(\conv(\cP_\cond), \{Q_0\}) = 0.
\end{equation}
This claim will be verified shortly. As a consequence, by Theorem
\ref{thm:kraft_le_cam}, any test which is level $\alpha$ over
\smash{$\cP_\cond$} must be \emph{powerless}, i.e., have power at most $\alpha$,
against the alternative $Q_0$ in \eqref{eq:alternative_q0}. Rephrasing this in
terms of prediction sets, using \eqref{eq:mu_type_II_x0}, we arrive at the
following hardness result.

\begin{theorem}[Conditional hardness, finite measure case]
\label{thm:hardness_finite}
Assume $\mu(\cY) < \infty$. Consider a mapping which takes $X_{n+1}$ and
$Z_1,\dots,Z_n$ to a prediction set $C_n(X_{n+1})$. If this method satisfies
conditional validity over \smash{$\cP_\iid$}, as in \eqref{eq:cond_validity1},
then for any \smash{$P_{X,Y} = P_X \times P_{Y|X}$} and any non-atom point 
$x_0 \in \supp(P_X)$, it holds that 
\[
\E_P[\mu(C_n(x_0))] \geq \mu(\cY) \cdot (1-\alpha). 
\]
In other words, the pointwise efficiency $\E_P[\mu(C_n(x_0))]$ of this method is
no better than the trivial method which draws $B \sim \mathrm{Bern}(1-\alpha)$ 
independently of everything else, and outputs $\cY$ if $B=1$ and $\emptyset$ 
otherwise.  
\end{theorem}

\begin{proof}
As before, we write $C_n(X_{n+1}) = \{ y : \phi(Z_1,\dots,Z_n,(X_{n+1},y)) = 0 
\}$ for a test $\phi$, which is without a loss of generality as we may always  
take $\phi(Z) = 1\{Y_{n+1} \notin C_n(X_{n+1})\}$. The assumed conditional
validity property \eqref{eq:cond_validity1} is equivalent to the marginal
property \eqref{eq:cond_validity2} over \smash{$\cP_\cond$}, which means $\phi$ 
is level $\alpha$ over \smash{$\cP_\cond$}. Combining \eqref{eq:mu_type_II_x0}
with Theorem \ref{thm:kraft_le_cam}, we have     
\[
\E_P[\mu(C_n(x_0))] \geq \mu(\cY) \cdot (1-\alpha) - \mu(\cY) \cdot
\TV(\conv(\cP_\cond), \{Q_0\}).
\]
Therefore, to prove the theorem, it suffices to verify \eqref{eq:tv_zero}. To
this end, define    
\[
\tilde{P}_{Y|X=x} = 
\begin{cases}
P_{Y|X=x} & \text{if $x \not= x_0$}, \\
\frac{\mu}{\mu(\cY)} & \text{if $x = x_0$},
\end{cases}
\]
and \smash{$P_0 = (P_X \times \tilde{P}_{Y|X})^n \times (\delta_{x_0} \times 
  \frac{\mu}{\mu(\cY)})$}. Then \smash{$P_0 \in \cP_\cond$}, and yet $\TV(P_0, 
Q_0) = 0$, since  
\begin{align*}
\TV(P_0, Q_0) 
&\leq n \cdot \TV(P_X \times \tilde{P}_{Y|X}, P_X \times P_{Y|X}) \\
&= n \cdot \sup_S \bigg| \int_S \d{P}_X(x) \, \d\tilde{P}_{Y|X=x}(y) - 
\int_S \d{P}_X(x) \, \d{P}_{Y|X=x}(y) \bigg| \\
&= n \cdot \sup_S \bigg| \int_{S \setminus \{x_0\}} \d{P}_X(x) \,
\d\tilde{P}_{Y|X=x}(y) - \int_{S \setminus \{x_0\} } \d{P}_X(x) \,
\d{P}_{Y|X=x}(y) \bigg| \\   
&= 0,
\end{align*}
where the second-to-last line holds because $x_0$ is a non-atom point of $P_X$,
and the last line holds because \smash{$\tilde{P}_{Y|X=x} = P_{Y|X=x}$} for $x
\not= x_0$, by construction. This completes the proof of \eqref{eq:tv_zero}, and
the theorem.
\end{proof}

\subsection{Infinite measure case}

In the infinite measure case $\mu(\cY) = \infty$, we can no longer normalize
$\mu$ over all of $\cY$ to form the distribution in
\eqref{eq:alternative_q0}. However, assuming that $\mu$ is $\sigma$-finite, 
which is true of Lebesgue measure on $\cY = \R$, we can adapt Theorem
\ref{thm:hardness_finite} using a truncation to reproduce a well-known
impossibility result from the conformal literature (cf.\ Proposition
4 in \citealt{vovk2012conditional}, or Lemma 1 in \citealt{lei2014distribution}).   

\begin{corollary}[Conditional hardness, infinite measure case]
\label{cor:hardness_infinite}
Assume $\mu(\cY) = \infty$ and assume $\mu$ is $\sigma$-finite. Consider a  
mapping which takes $X_{n+1}$ and $Z_1,\dots,Z_n$ to a prediction set
$C_n(X_{n+1})$. If this method satisfies conditional validity over
\smash{$\cP_\iid$}, as in \eqref{eq:cond_validity1}, then for any
\smash{$P_{X,Y} = P_X \times P_{Y|X}$} and any non-atom point $x_0 \in
\supp(P_X)$, it holds that  
\[
\E_P[\mu(C_n(x_0))] = \infty.
\]
\end{corollary}

\begin{proof}
Because $\mu$ is $\sigma$-finite and $\mu(\cY) = \infty$, there exists a
sequence of measurable sets $A_k \subseteq \cY$ such that $0 < \mu(A_k) <
\infty$ for all $k$ and $\mu(A_k) \to \infty$ as $k \to \infty$. For each $k$,
define the probability measure $\nu_k$ on $\cY$ by \smash{$\nu_k(B) =
  \frac{\mu(B \cap A_k)}{\mu(A_k)}$}. Then by Theorem \ref{thm:hardness_finite}
applied to $\nu_k$, we have $\E_P[\nu_k(C_n(x_0))] \geq 1-\alpha$, or
equivalently 
\[
\E_P[\mu(C_n(x_0) \cap A_k)] \geq (1-\alpha) \cdot \mu(A_k).
\] 
By monotonicity of $\mu$, we have $\mu(C_n(x_0)) \geq \mu(C_n(x_0) \cap A_k)$, 
hence 
\[
\E_P[\mu(C_n(x_0))] \geq (1-\alpha) \cdot \mu(A_k).
\]
Sending $k \to \infty$ completes the proof. 
\end{proof}

\section{Optimality}
\label{sec:optimality}

We examine another application of the representation for efficiency from Section 
\ref{sec:efficiency}, to study the question of optimality: characterizing the
precise form of optimally efficient methods for prediction sets, among all
methods with valid coverage over the class \smash{$\cP_\exch = \{ P : \text{$P$
    is exchangeable} \}$}. As explained in the discussion following
\eqref{eq:mu_type_II}, this question can be reduced to that of optimal power
among all tests which are valid over \smash{$\cP_\exch$}. Naturally,
Neyman-Pearson theory comes to mind, and this will indeed be our main tool,
after we carefully setting up and reformulating the problem, as described below.     

For a fixed \smash{$P \in \cP_\iid = \{ P : P = P_{X,Y}^{n+1} \;\, \text{for
    some $P_{X,Y} = P_X \times P_{Y|X}$} \}$}, we consider 
\begin{equation}
\begin{alignedat}{2}
\label{eq:optimality1}
&\minimize \quad &&\E_P[\mu(C_n(X_{n+1})] \\
&\st \quad &&\P_{P'}\{Y_{n+1} \in C_n(X_{n+1})\} \geq 1-\alpha, \; \text{for all
  $P' \in \cP_\exch$}, 
\end{alignedat}
\end{equation}
where the minimization is understood to be over mappings which take
$X_{n+1},Z_1,\dots,Z_n$ to a set $C_n(X_{n+1})$. Any such mapping can be written
as $C_n(X_{n+1}) = \{ y : \phi(Z_1,\dots,Z_n,(X_{n+1},y)) = 0 \}$ for a test
$\phi$ (we may always take $\phi(Z) = 1\{Y_{n+1} \notin
C_n(X_{n+1})\}$). Assuming $\mu(\cY) < \infty$, we now recall the formulation
\eqref{eq:mu_type_II} for the expected measure of the prediction set, with the 
alternative $Q$ defined in \eqref{eq:alternative_q}. The optimization
\eqref{eq:optimality1} is hence equivalent to
\begin{equation}
\begin{alignedat}{2}
\label{eq:optimality2}
&\maximize \quad &&\E_Q[\phi(Z)] \\
&\st \quad &&\E_{P'}[\phi(Z)] \leq \alpha, \; \text{for all $P' \in \cP_\exch$}.   
\end{alignedat}
\end{equation}
We develop one more reformulation. Writing \smash{$\E_Q[\phi(Z)] = \int
  \E_Q[\phi(Z) | U=u] \, \d{Q}_U(u)$} for \smash{$U = \Zbag$}, consider
optimizing the integrand for a fixed \smash{$u = \zbag$},  
\begin{equation}
\begin{alignedat}{2}
\label{eq:optimality3}
&\maximize \quad &&\E_{Q_u}[\phi(Z)] \\
&\st \quad &&\E_{P'}[\phi(Z)] \leq \alpha, \; \text{for all $P' \in \cP_\exch$},
\end{alignedat}
\end{equation}
where $Q_u$ denotes the distribution of $Z|U=u$ when $Z \sim Q$. If the
optimal test in \eqref{eq:optimality3} does not depend on $u$, then it will also
be optimal for \eqref{eq:optimality2}. This will indeed be the case and we will
focus on \eqref{eq:optimality3} henceforth. 

At this point, we are well-positioned to invoke Neyman-Pearson theory. Because
\eqref{eq:optimality3} involves a composite null hypothesis, we must first
establish a least favorable distribution. Once established, the Neyman-Pearson
lemma characterizes optimal tests. This approach traces back to the celebrated
work of \citet{neyman1933problem, neyman1933testing}, with extensions to
composite hypotheses and least favorable distributions by
\citet{wald1939contributions, lehmann1948most, wald1950statistical}. We state a 
version from \citet{lehmann2022testing}.

\begin{theorem}[Theorems 3.2.1 and 3.8.1 in \citealp{lehmann2022testing}] 
\label{thm:least_favorable}
Let $\cP$ be a class of null distributions and let $Q$ be a point
alternative. For a distribution $\Lambda$ over the class $\cP$, and define
the point null \smash{$P_\Lambda$} by \smash{$P_\Lambda(S) = \int_\cP P(S) \, 
  \d\Lambda(P)$}. Let \smash{$p_\Lambda$} and $q$ denote the densities
(Radon-Nikodym derivatives) of \smash{$P_\Lambda$} and $Q$, respectively, with
respect to a common base measure. Consider a test \smash{$\phi_\Lambda$} 
which satisfies  
\[
\phi_\Lambda(z) = 
\begin{cases}
1 & \text{if $q(z) > k p_\Lambda(z)$}, \\
0 & \text{if $q(z) < k p_\Lambda(z)$}, 
\end{cases}
\] 
and suppose $k$ can be chosen such that \smash{$\E_{P_\Lambda}[\phi_\Lambda(Z)]
  = \sup_{P \in \cP} \E_P[\phi_\Lambda(Z)] = \alpha$}. Then $\phi_\Lambda$
achieves the highest power against $Q$ among all tests which are level over
$\cP$, and the distribution $\Lambda$ is called \emph{least favorable}. 
\end{theorem}

\subsection{Optimal conformal score}

Our next step is to establish a least favorable distribution for testing the 
composite null \smash{$H_0 : P' \in \cP_\exch$} versus the point alternative 
$H_1 : P' = Q_u$, where recall \smash{$u = \zbag$} and $Q_u$ denotes the
distribution of $Z|U=u$ when $Z \sim Q$. First, we work out $Q_u$
explicitly. Let \smash{$p_{X,Y}$} be the density of \smash{$P_{X,Y}$} with
respect to $\rho \times \mu$, for some measure $\rho$ on $\cX$. Factorize
\smash{$p_{X,Y} = p_X \times p_{Y|X}$}. Then $Q_u$ is supported on permutations
of $z$, with   
\[
\P_{Q_u}\{ Z = z_\sigma \} = 
\frac{\prod_{i=1}^n p_{X,Y}(z_{\sigma(i)}) \cdot p_X(x_{\sigma(n+1)})}
{\sum_{\pi \in \cS_{n+1}} \prod_{i=1}^n p_{X,Y}(z_{\pi(i)}) \cdot
  p_X(x_{\pi(n+1)})}. 
\]
If we multiply and divide in the numerator by
\smash{$p_{Y|X=x_{\sigma(n+1)}}(y_{\sigma(n+1)})$}, and multiply and divide in
the denominator by \smash{$p_{Y|X=x_{\pi(n+1)}}(y_{\pi(n+1)})$}, then observe
that the numerator and denominator share a common factor of
\smash{$\prod_{i=1}^{n+1} p_{X,Y}(x_i,y_i)$}, which cancels, we are left with   
\begin{equation}
\label{eq:alternative_qu}
\P_{Q_u}\{ Z = z_\sigma \} =
\frac{c_u}{p_{Y|X=x_{\sigma(n+1)}}(y_{\sigma(n+1)})},  
\end{equation}
where \smash{$c_u = \big[ \sum_{\pi \in \cS_{n+1}}
  \frac{1}{p_{Y|X=x_{\pi(n+1)}}(y_{\pi(n+1)})} \big]^{-1}$} is a normalizing
constant.  

Consider now the choice \smash{$P'_u = \frac{1}{n+1} \sum_{\sigma \in \cS_{n+1}}
  \delta_{z_\sigma} \in \cP_\exch$}. We will show that $P'_u$ is least favorable
relative to $Q_u$. (By this, we mean that the distribution $\Lambda$ which
places all of its mass on $P'_u$ is least favorable.) Note that $P'_u$ and $Q_u$
are absolutely continuous with respect to counting measure on the multiset
\smash{$u = \zbag$}. We denote their densities by $p'_u$ and $q_u$,
respectively. Then the Neyman-Pearson optimal test for $P'_u$ versus $Q_u$ 
takes the form   
\[
\phi(z_\sigma) = 
\begin{cases}
1 & \text{if $q_u(z_\sigma) > k_u p'_u(z_\sigma)$}, \\ 
0 & \text{if $q_u(z_\sigma) < k_u p'_u(z_\sigma)$},
\end{cases}
\] 
for a threshold $k_u$. Using $p'_u(z_\sigma) = 1$, and the form of $q_u$ as
specified by \eqref{eq:alternative_qu}, this can be rewritten as
\begin{equation}
\label{eq:neyman_pearson}
\phi(z_\sigma) = 
\begin{cases}
1 & \text{if $1/p_{Y|X=x_{\sigma(n+1)}}(y_{\sigma(n+1)}) > k_u$}, \\ 
0 & \text{if $1/p_{Y|X=x_{\sigma(n+1)}}(y_{\sigma(n+1)}) < k_u$},
\end{cases}
\end{equation}
for a redefined threshold $k_u$. This threshold is chosen so that
\begin{equation}
\label{eq:exact_type_I}
\E_{P'_u}[\phi(Z)] = \frac{1}{(n+1)!} \sum_{\sigma \in \cS_{n+1}} \phi(z_\sigma)
= \alpha.  
\end{equation}
The key observation is these two conditions \eqref{eq:neyman_pearson},  
\eqref{eq:exact_type_I} are fulfilled by a randomized permutation test, as
explained below, and proved in Appendix \ref{app:phi_star}.

\begin{proposition}
\label{prop:phi_star}
Let \smash{$T(Z) =  1/p_{Y|X=X_{n+1}}(Y_{n+1})$}, and define the test
\smash{$\phi^*(Z) = 1\{p^* \leq \alpha\}$}, where
\[
p^* = \frac{1}{(n+1)!} \sum_{\sigma \in \cS_{n+1}} \Big( 1\{ T(Z_\sigma) >
T(Z) \} + V \cdot 1\{ T(Z_\sigma) = T(Z) \} \Big),
\]
for $V \sim \mathrm{Unif}[0,1]$, independently of everything else. Then $\phi^*$
satisfies \eqref{eq:neyman_pearson}, \eqref{eq:exact_type_I}.
\end{proposition}

As the Neyman-Pearson optimal test for $P'_u$ versus $Q_u$ takes  
the form of a randomized permutation test $\phi^*$, as defined in Proposition  
\ref{prop:phi_star}, we know that it controls type I error over all of
$\cP_\exch$ (as do all randomized permutation tests; see, e.g., Theorem 17.2.1
in \citealt{lehmann2022testing}). Thus, we have succeeded in identifying a
least favorable distribution, and by Theorem \ref{thm:least_favorable}, we know
that $\phi^*$ solves \eqref{eq:optimality3}. Finally, as $\phi^*$ does not
depend on $u$ (note the choice of test statistic \smash{$T(Z) =
  1/p_{Y|X=X_{n+1}}(Y_{n+1})$} depends only on $Q$, not $Q_u$), we conclude that
$\phi^*$ solves \eqref{eq:optimality2}. Transcribing this back to prediction
sets yields the following result. 

\begin{theorem}[Optimality]
\label{thm:optimality}
Assume $\mu$ is a $\sigma$-finite measure on $\cY$. Fix any
distribution \smash{$P_{X,Y} = P_X \times P_{Y|X}$}, denote by
\smash{$p_{Y|X=x}$} the density of \smash{$P_{Y|X=x}$} with respect to $\mu$,
and assume that \smash{$p_{Y|X=x}(y) > 0$} for all $y \in \cY$, and $P_X$-almost
every $x$. Let \smash{$P = P_{X,Y}^{n+1}$}. Among all mappings which take 
$X_{n+1},Z_1,\dots,Z_n$ to a prediction set $C_n(X_{n+1})$, and provide
$1-\alpha$ coverage over the class \smash{$\cP_\exch$} of exchangeable
distributions, the optimal efficiency $\E_P[\mu(C_n(X_{n+1})]$ is attained by
the following randomized conformal predictor:   
\begin{equation}
\label{eq:cp_optimal_set}
C^*_n(X_{n+1}) = \bigg\{ y : \frac{1}{n+1} \sum_{i=1}^{n+1} \Big( 1\{ s^y_i >
s^y_{n+1} \} + V \cdot 1\{ s^y_i = s^y_{n+1} \} \Big) > \alpha \bigg\},
\end{equation}
where $V \sim \mathrm{Unif}[0,1]$, independently of everything else, the scores
are as in \eqref{eq:cp_score}, and the score function is 
\begin{equation}
\label{eq:cp_optimal_score}
s(z_{n+1}; z) = \frac{1}{p_{Y|X=x_{n+1}}(y_{n+1})}.
\end{equation} 
In other words, the construction given in \eqref{eq:cp_optimal_set},
\eqref{eq:cp_optimal_score} solves problem \eqref{eq:optimality1}.    
\end{theorem}

\begin{proof}
Assume first that $\mu(\cY) < \infty$. Then the distribution $Q$ in
\eqref{eq:alternative_q} is well-defined and the arguments above establish that
the test $\phi^*$ in Proposition \ref{prop:phi_star} solves
\eqref{eq:optimality2}. The corresponding prediction set is  
\[
C^*_n(X_{n+1}) = \bigg\{ y : \frac{1}{(n+1)!} \sum_{\sigma \in \cS_{n+1}} \Big(
1\{ T(Z_\sigma) > T(Z) \} + V \cdot 1\{ T(Z_\sigma) = T(Z) \} \Big) > \alpha
\bigg\}, 
\]   
which reduces to the claimed form \eqref{eq:cp_optimal_set} with score
$s(z_{n+1}; z) = T(z)$ as in \eqref{eq:cp_optimal_score}, because this score
function is invariant to relabeling $z_1,\dots,z_n$. This completes the proof
for the finite measure case. 

For the general case, permitting infinite measure $\mu(\cY) = \infty$, we can 
adapt the above strategy by using a slightly different perspective which does
not require defining $Q$ as in \eqref{eq:alternative_q}, deferred to Appendix 
\ref{app:optimality}. 
\end{proof}

Several closely related results exist in the literature, as we describe
below. Thus, we emphasize the above result not for its novelty, but as a clear
demonstration of the power of the hypothesis-theoretic framework for prediction
sets. We also remark that the choice of an i.i.d.\ reference distribution
\smash{$P = P_{X,Y}^{n+1}$} with which to measure efficiency is made for
simplicity. The arguments can be extended to an exchangeable distribution $P$,
which leads to a somewhat more complex choice of optimal score: the inverse
conditional density of $Y_{n+1}$ given $X_{n+1}$ and the rest of the samples
$Z_1,\dots,Z_n$.   

\subsection{Related results}

There are numerous related results in the conformal prediction literature which
characterize the optimal score at the population-level; these results
effectively take $n \to \infty$ (formalized in different ways), and in this
large-sample limit, they all arrive at the same answer: the optimal score is the
inverse conditional density
\smash{$1/p_{Y|X=x_{n+1}}(y_{n+1})$}. This is covered in Chapter 3.1 of
\citet{vovk2022algorithmic}, and Chapters 5.2 and 5.3 of
\citet{angelopoulos2025theoretical}; see also \citet{lei2013distribution,
  lei2018distribution, sadinle2019least, izbicki2020flexible} for related oracle
constructions. The key distinction is that our result in Theorem
\ref{thm:optimality} is finite-sample, i.e., it is about conformal prediction
itself, not a large-sample analog. 

In terms of finite-sample results, \citet{dandapanthula2025offline,
  hore2026conformal} developed what they call the first and second ``conformal
Neyman-Pearson lemmas'' for confidence sets in nonparametric changepoint
detection. These results are derived by reducing the question of an optimal
score function to one of optimal power in testing, and applying the
Neyman-Pearson lemma. While their problem setting is different than
ours---changepoint detection over i.i.d.\ sequences, with tests of a point null
versus point alternative---the spirit of the construction is similar.   

For prediction sets, versions of the finite-sample result in Theorem  
\ref{thm:optimality} appear in recent literature studying optimality: Theorem
4.1 of \citet{hoff2023bayes} and Theorem 8 of \citet{koning2025fuzzy}. These
authors use different approaches, which we discuss
separately. \citet{hoff2023bayes} develops a general framework for prediction
sets, attributed originally to \citet{faulkenberry1973method}, which at its core
is based on conditioning on a complete sufficient statistic. This is closely
related to the perspectives on Neyman structure and completeness in Section
\ref{sec:universality} of this paper. There are some superficial differences in
Hoff's Theorem 4.1 and our Theorem \ref{thm:optimality}: their result does not
include covariates, and they analyze a Bayesian notion of efficiency, 
with respect to a mixture distribution $P_\pi$ defined by a prior $\pi$ over
$\cP$. In our view, these are minor and one can adapt their framework
to handle covariates, and to analyze frequentist efficiency (by taking $\pi$ to
be a point mass). 

A bigger difference lies in the apparent generality of Hoff's conclusion.
Theorem 4.1 in \citet{hoff2023bayes} derives a form for the optimal conformal
score based on the posterior predictive density (optimal among conformal
methods), but falls short of concluding that the corresponding conformal
predictor is optimal among all valid methods for prediction sets. We suspect
that this is due to their choice of constraint set: they consider methods which
obtain exact coverage over the class \smash{$\cP_\iid$} of i.i.d.\
distributions, instead of methods which have valid coverage over the class
\smash{$\cP_\exch$} of exchangeable distributions, as we do in
\eqref{eq:optimality1}. Hoff's focus on exact coverage stems from their use of
the traditional ``two-sided'' definition of bounded completeness
\eqref{eq:boundedly_complete_2s}, which is tied to exact type I error
control. We introduce a ``one-sided'' notion of bounded completeness
\eqref{eq:boundedly_complete_1s}, as a tool for analyzing valid type I error
control. For i.i.d.\ distributions, there is a gap between these two
definitions: the multiset \smash{$U = \Zbag$} is boundedly complete according to
\eqref{eq:boundedly_complete_2s} but not \eqref{eq:boundedly_complete_1s} (see
Appendix \ref{app:boundedly_complete}).  In other words, for i.i.d.\
distributions, there seems to be a fundamental difference in whether
multiset-conditioning is necessary for exact versus valid (but possibly
conservative) type I error control.

\citet{koning2025fuzzy} construct prediction sets by leveraging the duality to
testing, as in our paper. Their primary focus is on ``fuzzy'' prediction sets,
which generalize the binary nature (inclusion/exclusion) of a traditional
prediction set, and can be seen as rescaled e-values. They observe that 
efficiency of a traditional prediction set can be represented in terms of type
II error/power against a particular alternative, as in Section
\ref{sec:efficiency} of this paper; for fuzzy prediction sets, they generalize
this by defining a utility function used to measure the amount of evidence
presented by the e-value under the alternative. Leveraging tools from the
emerging literature on e-values \citep{koning2023measuring,
  koning2024continuous}, they first reduce validity over the class
\smash{$\cP_\exch$} of exchangeable null distributions to multiset-conditional
validity, then derive utility-optimal fuzzy prediction sets for concave,
nondecreasing, and upper-semicontinuous utility functions. For the
``Neyman-Pearson utility'', the consequence in Theorem 8 and Corollary 2 of 
\citet{koning2025fuzzy} is that conformal prediction with inverse density score
\smash{$1/p_{Y|X}$} is optimally efficient among all methods with valid coverage
over \smash{$\cP_\exch$}. Theorem \ref{thm:optimality} in our paper reproduces
this conclusion with a simpler argument based on classical Neyman-Pearson
theory (and technically, extends this result to accommodate a $\sigma$-finite 
measure $\mu$ used to measure efficiency, such as Lebesgue measure on $\cY =
\R$).    

\subsection{Interpretations and illustrations}

The interpretation of the result in Theorem \ref{thm:optimality} is that
efficient conformal methods should approximate the density \smash{$p_{Y|X}$} as
closely as possible. This helps us contextualize a choice such as the absolute 
residual score from a predictor of the conditional mean $f(x) =
\E_P[Y|X=x]$. Consider a model in which    
\begin{equation}
\label{eq:homoskedastic}
1/p_{Y|X=x}(y) \;\, \text{is a monotone increasing function of} \;\, |y -
f(x)|,
\end{equation}
such as the conditional normal model $Y|X=x \sim N(f(x), \sigma^2)$. Because
conformal sets are invariant under monotone transformations of the score, we can
view the use of a conditional mean predictor \smash{$\hf = A(z)$} (from an
algorithm $A$) together with residual score \smash{$s(z_{n+1}; z) = |y_{n+1} -
  \hf(x_{n+1})|$} as a plug-in strategy for estimating the conditional density
under \eqref{eq:homoskedastic}.           

The model \eqref{eq:homoskedastic} assumes homoskedasticity. If we instead
consider the heteroskedastic conditional normal model $Y|X=x \sim N(f(x),
\sigma^2(x))$, then 
\begin{equation} 
\label{eq:heteroskedastic}
1/p_{Y|X=x}(y) \;\, \text{is a monotone increasing function of} \;\, \frac{(y -
  f(x))^2}{2\sigma^2(x)} + \log \sigma(x). 
\end{equation}
It is interesting to note that \eqref{eq:heteroskedastic} does \emph{not}
lead to studentized score, but something different---it suggests forming 
conditional mean \smash{$\hf = A(z)$} and variance \smash{$\hat\sigma = B(z)$}  
predictors (from algorithms $A,B$), and using
\begin{equation}
\label{eq:efficient_score}
s(z_{n+1}; z) = \frac{(y_{n+1} - \hf(x_{n+1}))^2}{2\hat\sigma^2(x_{n+1})} +
\log \hat\sigma(x_{n+1}), 
\end{equation}
as the corresponding plug-in strategy under
\eqref{eq:heteroskedastic}. Studentized score would keep only the first term
above, or equivalently (again due to invariance under monotone transformations),   
\begin{equation}
\label{eq:student_score}
s(z_{n+1}; z) = \frac{|y_{n+1} - \hf(x_{n+1})|}{\hat\sigma(x_{n+1})}.
\end{equation}
This score has its own benefits (e.g., when \smash{$\hf,\hat\sigma$} are
accurate it brings us closer to conditional coverage, as studied in
\citealt{lei2018distribution}), but is inefficient compared to
\eqref{eq:efficient_score}. The score \eqref{eq:efficient_score} shrinks the 
prediction sets in high-variance regions compared to the studentized score
\eqref{eq:student_score}, due to the second term \smash{$\log
  \hat\sigma(x_{n+1})$}, which sacrifices local coverage but improves
efficiency. Figure \ref{fig:efficient_score} provides an illustration.  

\begin{figure}[H]
\centering
\includegraphics[width=0.675\textwidth]{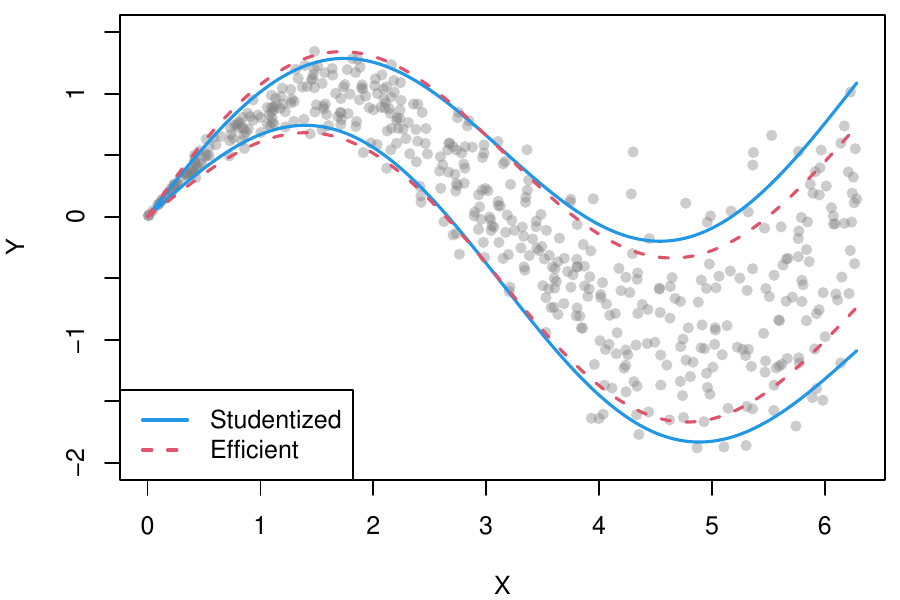}
\caption{Simple example of conformal prediction intervals using studentized
  \eqref{eq:student_score} and efficient \eqref{eq:efficient_score} scores in
  the Gaussian heteroskedastic model.  See Appendix \ref{app:efficient_score}
  for further explanation and discussion.}
\label{fig:efficient_score}
\end{figure}

\section{Discussion} 

This paper presents a perspective of conformal prediction methods as inverting
tests of exchangeability. This view allows us to reinterpret well-known 
conformal universality and impossibility results as consequences of hypothesis
testing theory from the 1950s. It also facilitates the construction of optimal
conformal predictors, as an application of Neyman-Pearson theory. 

Broadly, the goal of this work is to motivate and document the direct relevance
of foundational statistical inference---predominantly developed in the first
half of the 20th century---to modern predictive inference. While we studied 
standard conformal prediction (for exchangeable samples), the relevance of
hypothesis-theoretic tools should extend beyond the standard setting, as many
newer variants of conformal can be seen as inverting hypothesis tests given
partial information about the data \citep{barber2026unifying}. Much ground
remains to be explored in both directions: what hypothesis testing can teach us
about conformal prediction, and vice versa. 

\section*{Acknowledgments}

We thank Emmanuel Cand{\`e}s for helpful comments, and Nick Koning for
introducing us to his work. R.J.T.\ was supported by the Office of Naval
Research (ONR) grant N00014-20-1-2787, and R.F.B.\ by ONR grant
N00014-24-1-2544. 

An AI model was used to brainstorm the key representation in
\eqref{eq:mu_type_II} of efficiency in terms of type II error. AI was also used
to help identify references, and to implement the example in Figure
\ref{fig:efficient_score}. 

\bibliographystyle{plainnat}
\bibliography{main}

\clearpage
\appendix

\section{Appendix}

\subsection{Proof of Theorem \ref{thm:neyman_structure}}
\label{app:neyman_structure}

Let $\phi$ be a test that is valid over $\cP$, according to 
\eqref{eq:validity}. Let $f(u) = \E[\phi(Z) | U=u] - \alpha$, which by
sufficiency of $U$, does not depend on $P$. Then for any $P_U \in \cP_U$, we
have \smash{$\E_{P_U}[f(U)] = \E_P[\phi(Z)] - \alpha \leq 0$}, by
assumption. Furthermore, $f$ is bounded (by $1-\alpha$), so by bounded
completeness of $U$ with respect to $\cP_U$, it follows that $f(u) \leq 0$ for
$\cP_U$-almost every $u$, which is equivalent to \eqref{eq:neyman_structure}. 

\subsection{Bounded completeness}
\label{app:boundedly_complete}

The traditional definition of bounded completeness of a sufficient statistic $U$ 
for a class $\cP$ states that, for any bounded function $f$,   
\begin{equation}
\label{eq:boundedly_complete_2s}
\E_{P_U}[f(U)] = 0 \;\, \text{for all $P \in \cP$} \quad \implies \quad 
f(u) = 0 \;\, \text{for $\cP_U$-almost every $u$}.  
\end{equation}
In our version \eqref{eq:boundedly_complete_1s}, we have replaced equalities
with inequalities. Our version may be seen as ``one-sided'' bounded
completeness, and the traditional version as ``two-sided'' bounded completeness.    

The one-sided version \eqref{eq:boundedly_complete_1s} is strictly stronger than
the traditional two-sided version \eqref{eq:boundedly_complete_2s}. To see this,
note that we can apply \eqref{eq:boundedly_complete_1s} to both $f$ and $-f$ to
conclude \eqref{eq:boundedly_complete_2s}. However, the reverse implication is
not true.   

To demonstrate this, we first recall that it is well-known that \smash{$U =
  \Zbag$} is boundedly complete in the two-sided sense
\eqref{eq:boundedly_complete_2s} with respect to the class of i.i.d.\ continuous
distributions \citep{fraser1954completeness}. This extends to some discrete
distributions, such as the class of i.i.d.\ Bernoulli distributions for binary
random variables \citep{lehmann1950completeness,
  lehmann1955completeness}. However, consider the following example: draw
i.i.d.\ $Z_1,Z_2 \sim \mathrm{Bern}(p)$ and enumerate the distinct multisets
by \smash{$u_1=\lbag 0,0 \rbag$}, \smash{$u_2 = \lbag 1,0 \rbag$}, and
\smash{$u_3 = \lbag 1,1 \rbag$}. Let $f(u_1) = -1$, $f(u_2) = 1$, and $f(u_3) =
-1$. Then  
\[
\E_p[f(U)] = -(1-p)^2 + 2p(1-p) - p^2 = -(2p-1)^2 \leq 0, \quad \text{for all $p
  \in [0,1]$}. 
\]
However, $f(u_2) > 0$, which contradicts one-sided bounded completeness
\eqref{eq:boundedly_complete_1s}. 

The above example also reveals an interesting difference in the ``richness'' 
of the classes \smash{$\cP_\iid$} and \smash{$\cP_\exch$} of i.i.d.\ and 
exchangeable distributions, respectively. Recall, the arguments in Section
\ref{sec:universality} establish that \smash{$\cP_\exch$} is rich enough to
admit \smash{$U = \Zbag$} as a boundedly complete statistic according
to the one-sided definition in \eqref{eq:boundedly_complete_1s}. But the
above example shows that this is not true for \smash{$\cP_\iid$}, and 
\smash{$U = \Zbag$} is only boundedly complete according to the weaker 
two-sided definition in \eqref{eq:boundedly_complete_2s}.  

\subsection{Proof of claim \eqref{eq:neyman_structure_perm}}
\label{app:neyman_structure_perm}

We assume $\alpha < 1$, as otherwise the claim is vacuous. Note first that
the ``$\impliedby$'' direction of the claim is due to the validity of
permutation tests over each exchangeable distribution of the form
\smash{$\frac{1}{(n+1)!} \sum_{\sigma \in \cS_{n+1}} \delta_{z_\sigma}$}. 

For the ``$\implies$'' direction of the claim, let $\phi$ satisfy
\eqref{eq:neyman_structure_bag}. Define a test statistic by $T = \phi$ and
permutation p-value $p$ by \eqref{eq:perm_pval}. We will show that under this
construction $\phi(Z) = 1\{p \leq \alpha\}$, and proceed to check this in two
cases. 
\begin{itemize}
\item If $\phi(Z) = 1$, then 
\[
p = \frac{1}{(n+1)!} \sum_{\sigma \in \cS_{n+1}} 1\{ \phi(Z_\sigma) \geq 1 \}
= \frac{1}{(n+1)!} \sum_{\sigma \in \cS_{n+1}} \phi(Z_\sigma) \leq \alpha.
\]
\item If $\phi(Z) = 0$, then 
\[
p = \frac{1}{(n+1)!} \sum_{\sigma \in \cS_{n+1}} 1\{ \phi(Z_\sigma) \geq 0 \} =
1 > \alpha.
\]
\end{itemize}
This completes the proof.

\subsection{Proof of Proposition \ref{prop:phi_star}}
\label{app:phi_star}

Following Chapter 17.2.1 in \citet{lehmann2022testing}, let \smash{$T_{(1)} \leq 
  \cdots \leq T_{(N)}$} denote the ordered values of $T(z_\sigma)$ over $\sigma
\in \cS_{n+1}$, where $N = (n+1)!$. Let \smash{$k = N - \lfloor N\alpha
  \rfloor$}, where \smash{$\lfloor N\alpha \rfloor$} is the largest integer less
than or equal to $N\alpha$, and define   
\[
\phi(z_\sigma) = 
\begin{cases}
1 & \text{if $T(z_\sigma) > T_{(k)}$} \\
0 & \text{if $T(z_\sigma) < T_{(k)}$}.
\end{cases}
\]
We see that \eqref{eq:neyman_pearson} is met for \smash{$k_u = T_{(k)}$}. 
Further, for \smash{$N_+ = |\{i : T_{(i)} > T_{(k)}\}|$} and \smash{$N_= =|\{i 
  : T_{(i)} = T_{(k)}\}|$}, set     
\[
\phi(z_\sigma) = \frac{N\alpha - N_+}{N_=} \quad \text{if $T(z_\sigma) = T_{(k)}$}.
\]
Then by Theorem 17.2.1 in \citet{lehmann2022testing}, the test $\phi$ has level
exactly $\alpha$, as in \eqref{eq:exact_type_I}.   

It remains to show that $\phi$ defined above is equivalent to the test $\phi^*$
defined in the proposition; precisely, we will show $\phi = \E_V[\phi^*]$, where
the expectation is over $V$. For this, it helps to rewrite $\phi^*$ in terms of 
its action over $z_\sigma$, $\sigma \in \cS_{n+1}$. This is $\phi^*(z_\sigma) =
1\{p^*(z_\sigma) \leq \alpha\}$, where 
\[
p^*(z_\sigma) = \frac{N^\sigma_+ + V \cdot N^\sigma_=}{N},
\]
and \smash{$N^\sigma_+ = |\{ i : T_{(i)} > T(z_\sigma) \}|$}, \smash{$N^\sigma_= 
  = |\{ i : T_{(i)} = T(z_\sigma) \}|$}. We proceed in cases. Below ``almost
sure'' should be interpreted with respect to $V$.  
\begin{itemize}
\item If \smash{$N^\sigma_+ + N^\sigma_= \leq N\alpha$}, then $p^* \leq \alpha$ 
  almost surely and $\E_V[\phi^*(z_\sigma)] = 1$. In this case, we also know at
  most \smash{$\lfloor N\alpha \rfloor$} statistics are $\geq T(z_\sigma)$, or
  equivalently, more than \smash{$N - \lfloor N\alpha \rfloor = k$} statistics 
  are $< T(z_\sigma)$, thus \smash{$T(z_\sigma) > T_{(k)}$} and $\phi(z_\sigma)
  = 1$.    
\item If \smash{$N^\sigma_+ > N\alpha$}, then $p^* > \alpha$ almost surely and
  $\E_V[\phi^*(z_\sigma)] = 0$. Also, in this case, at least \smash{$\lfloor
    N\alpha \rfloor + 1 =$} $N-k+1$ statistics are $> T(z_\sigma)$, thus
  \smash{$T(z_\sigma) < T_{(k)}$} and $\phi(z_\sigma) = 0$.    
\item If \smash{$N^\sigma_+ \leq N\alpha < N^\sigma_+ + N^\sigma_= $}, then 
  \[
  \E_V[\phi^*(z_\sigma)] = \P\Big\{N^\sigma_+ + V \cdot N^\sigma_= \leq N\alpha
  \Big\} = \frac{N\alpha - N^\sigma_+}{N^\sigma_=}.
  \]
  In this case, $T(z_\sigma) = T_{(k)}$, so \smash{$N_+ = N^\sigma_+$},
  \smash{$N_= = N^\sigma_=$}, and finally $\phi(z_\sigma) =
  \E_V[\phi^*(z_\sigma)]$.   
\end{itemize}
This completes the proof.

\subsection{Proof of Theorem \ref{thm:optimality}, general case}
\label{app:optimality}

In the general case, where $\mu$ is $\sigma$-finite but not necessarily finite,
we can adapt the arguments presented in finite case, as follows. Recall the 
representation \eqref{eq:mu_integral}. Since \smash{$p_{Y|X=x}(y) > 0$} for all
$y$ and $P_X$-almost every $x$, we can rewrite this as 
\[
\E_P[\mu(C_n(X_{n+1}))] = \E_P \bigg[ (1 - \phi(Z)) \cdot
\frac{1}{p_{Y|X=X_{n+1}}(Y_{n+1})} \bigg].  
\]
Next, by the total law of expectation
\[
\E_P[\mu(C_n(X_{n+1}))] = \E_{P_U} \Bigg[ \underbrace{\E_P \bigg[ (1 - \phi(Z))
  \cdot \frac{1}{p_{Y|X=X_{n+1}}(Y_{n+1})} \,\bigg|\, U \bigg]}_{h(\phi,U)}
\Bigg],    
\]
where by exchangeability the conditional expectation is 
\[
h(\phi,u) = \frac{1}{(n+1)!} \sum_{\sigma \in \cS_{n+1}} (1 - \phi(z_\sigma))
\cdot \frac{1}{p_{Y|X=x_{\sigma(n+1)}}(y_{\sigma(n+1)})}.
\]
Fix \smash{$u = \zbag$}, and define $Q_u$ as in \eqref{eq:alternative_qu}. 
Observe that
\[
h(\phi,u) = \frac{1}{c_u (n+1)!} \cdot \E_{Q_u}[1-\phi(Z)].
\]
Hence, even in the general case (where $Q$ itself is not necessarily
well-defined), we have reduced the problem to one of minimizing type II error
against $Q_u$. By the same arguments as given earlier, the test $\phi^*$ in
Proposition \ref{prop:phi_star} has optimal type II error among all valid
level $\alpha$ tests for \smash{$\cP_\exch$} versus $Q_u$. That is, for
any test $\phi$ which is level $\alpha$ for \smash{$\cP_\exch$}, 
\[
h(\phi^*,u) \leq h(\phi,u), \quad \text{for all $u$}. 
\]
Furthermore, since $\phi^*$ does not depend on $u$, we can integrate over $U$ to
obtain  
\[
\E_P[\mu(C^*_n(X_{n+1}))] \leq \E_P[\mu(C_n(X_{n+1}))].
\]
This completes the proof.

\subsection{Further details for Figure \ref{fig:efficient_score}}
\label{app:efficient_score} 

We generated $n = 500$ i.i.d.\ samples $(X_i,Y_i)$, $i=1,\dots,n$, where each
$X_i \sim \mathrm{Unif}[0,2\pi]$, and   
\[
Y_i|X_i \sim N( f(X_i), \sigma^2(X_i)), \quad \text{where $f(x) = \sin(x)$ and
  $\sigma(x) = \pi x /30$}.
\]
These are the gray dots in Figure \ref{fig:efficient_score}. For both the
efficient score \eqref{eq:efficient_score} and studentized score
\eqref{eq:student_score}, we used the true mean and variance functions, i.e., 
\smash{$\hf = f$} and \smash{$\hat\sigma = \sigma$}. We then pursued the
conformal construction in \eqref{eq:cp_set}, \eqref{eq:cp_score} at the coverage
level $1-\alpha = 0.9$.   

In this setup, due to our use of oracle mean and variance functions (no model  
training), the score \smash{$s_i^y$} does not depend on $y$ for $i \leq n$, for
either studentized or efficient score functions. This allows us to compute the
conformal sets in closed-form. For studentized score \eqref{eq:student_score}, 
this is    
\[
C_n(x) = [ f(x) - \sigma(x) \hat{q}_n, \, f(x) + \sigma(x) \hat{q}_n ],
\]
where \smash{$\hat{q}_n$} is the $\lceil (n+1) (1-\alpha) \rceil\th$ smallest of 
$s_1,\dots,s_n$. For the efficient score \eqref{eq:efficient_score}, this is 
\[
C_n(x) = \Big[ f(x) - \sigma(x) \sqrt{2(\hat{q}_n - \log\sigma(x))}, \,  
f(x) +  \sigma(x) \sqrt{2(\hat{q}_n - \log\sigma(x))} \Big],
\]
where again \smash{$\hat{q}_n$} is the $\lceil (n+1) (1-\alpha) \rceil\th$
smallest of $s_1,\dots,s_n$. 

The form above makes it clear why, for the most part, the efficient intervals
(red) are narrower than the studentized intervals (blue) in Figure
\ref{fig:efficient_score}. This is true except in regions where $\sigma(x)$ is
very small, in which case $\log\sigma(x) < 0$. In this simulation, the
studentized intervals are on an average about 12.5\% wider than the efficient 
intervals.   
\end{document}